\documentclass[review]{amsart}
\usepackage{graphicx}
\newtheorem{thm}{Theorem}[section]

\newtheorem{lem}[thm]{Lemma}
\newtheorem{prop}[thm]{Proposition}
\theoremstyle{definition}

\theoremstyle{remark}
\newtheorem{rem}[thm]{Remark}
\numberwithin{equation}{section}
\newcommand{\Real}{{\bf R}}
\newcommand{\eps}{\varepsilon}

\def\R{{\bf R}}

\def\N{{\bf N}}

\def\S{{\bf S}}
\def\d{\displaystyle}
\def\e{{\varepsilon}}

\def\vp{\varphi}

\begin{document}

\title[]
{Blow-up of solutions to critical semilinear wave equations with variable coefficients}
\author{Kyouhei Wakasa and Borislav Yordanov}

\address{Department of Mathematics, Faculty of Science and Technology, 
Tokyo University of Science, 2641 Yamazaki, Noda-shi, Chiba, 278-8510, Japan}
\address{Office of International Affairs, Hokkaido University, 
Kita 15, Nishi 8, Kita-ku, Sapporo, Hokkaido 060-0815, Japan 
and Institute of Mathematics, Sofia}
\vskip10pt
\address{}
\email{wakasa\_kyouhei@ma.noda.tus.ac.jp}
\email{byordanov@oia.hokudai.ac.jp}

%\thanks{}%
%\subjclass{}%
%\keywords{}%

%\date{}%
%\dedicatory{}%
%\commby{*}%

\date{\today}
 \keywords{}

\begin{abstract}

We verify the critical case $p=p_0(n)$ of Strauss' conjecture~\cite{St81} concerning the blow-up of solutions to
semilinear wave equations with variable coefficients in $\Real^n$, where $n\geq 2$.
The perturbations of Laplace operator are assumed to be smooth and decay exponentially fast at infinity.
We also obtain a sharp lifespan upper bound for solutions with compactly supported data when $p=p_0(n)$.
The unified approach to blow-up problems in all dimensions combines several classical ideas in order to generalize and simplify
the method of Zhou~\cite{Z07} and Zhou $\&$ Han~\cite{ZH14}:
exponential ``eigenfunctions" of the Laplacian~\cite{YZ06} are used to construct the test function $\phi_q$
for linear wave equation with variable coefficients and John's method of iterations~\cite{J79} is augmented with the \lq\lq slicing method" of Agemi, Kurokawa and Takamura \cite{AKT00} for lower bounds in the critical case.
\end{abstract}
\maketitle

% ----------------------------------------------------------------

\section{Introduction}

We study the blow-up part of Strauss' conjecture~\cite{St81} in the case of semilinar wave equations with critical
nonlinearities and metric perturbations of the Laplacian
\[
\Delta_g =\sum_{i,j=1}^n\partial_{x_i}g_{ij}(x)\partial_{x_j},
\]
where $g=(g_{ij})\in C^\infty(\Real^n)$ satisfies the following: there exist $\gamma>0$ and $\beta>0$,
\begin{eqnarray}
\label{g}
\sum_{i,j=1}^n g_{ij}(x)\xi_i\xi_j & \geq &\gamma |\xi|^2,\quad         \xi\in\Real^n,\\
\label{g1}
\sum_{i,j=1}^n |\nabla g_{ij}(x)|+|g_{ij}(x)-\delta_{ij}| & = & O(e^{-\beta|x|}),\quad   |x|\rightarrow\infty.
\end{eqnarray}
The problem is to determine what range of $p>1$ allows some solutions of
\begin{equation}
\label{ivp}
\left\{
  \begin{array}{ll}
     u_{tt}-\Delta_g u=|u|^{p}, & x\in \Real^n,\quad t>0, \\
      (u,u_t)|_{t=0}=(\eps u_0,\eps u_1), &  x\in \Real^n,
  \end{array}
\right.
\end{equation}
with $(u_0,\: u_1)\in C_0^{\infty}(\R^n)\times C_0^{\infty}(\R^n)$, to blow up in finite time regardless of
any smallness condition on $\e>0$. It is also interesting to estimate the lifespan of such solutions as
$\eps\rightarrow 0,$ in order to verify the sharpness of results on almost global existence
obtained by other methods~\cite{W}. The history of these problems spans almost four decades
beginning with the work of Fritz John~\cite{J79} in 1979.

%\begin{eqnarray}
%\label{main}
% u_{tt}-\Delta_g u=|u|^{p},& & x\in \Real^n,\quad t>0,\\
%\label{data}
%u|_{t=0}=\eps u_0,\quad u_t|_{t=0}=\eps u_1,& & x\in \Real^n,
%\end{eqnarray}

%$\hbox{supp}\:(f,g) \subset \{|x|\le R_0\}$ with $R_0>0.$

When $g_{ij}(x)=\delta_{ij}$, the original conjecture of Walter Strauss~\cite{St81} reads as follows:
there exists a critical exponent $p_0(n)$, such that (\ref{ivp}) has a
global in time solution for sufficiently small $\eps>0$ if $p>p_0(n)$ and
(\ref{ivp}) has a solution that blows up in finite time for every $\eps>0$ if $1<p<p_0(n)$.
Actually, the Strauss' exponent $p_0(n)$ is defined as the positive root of the quadratic equation
$\gamma(p,n)=0$, where
\begin{equation}
\label{Strauss-eq}
\gamma(p,n)=2+(n+1)p-(n-1)p^2.
\end{equation}
This conjecture was first verified by John~\cite{J79} when $n=3$, except for $p=p_0(3)$.
Later, Glassey \cite{G81a}, \cite{G81b} established the conjecture when $n=2$, excluding again $p=p_0(2)$.
The critical cases $p=p_0(n)$ in $n=2$ and $3$ dimensions were shown by Schaeffer~\cite{Sc85}
to belong to the blow-up range.
In higher space dimensions $n\ge4$, Sideris \cite{Si84} verified the blow-up part for subcritical
$1<p<p_0(n)$. The proof was simplified by Rammaha~\cite{R88} and Jiao $\&$ Zhou~\cite{JZ}.
The global existence in the supercritical case $p>p_0(n)$ was proved by
Kubo~\cite{Kubo96} (radial case, odd dimensions),
Kubo $\& $ Kubota~\cite{KK98} (radial case, even dimensions), Zhou~\cite{Z95} ($n=4$) and finally Georgiev $\&$ Lindblad $\&$ Sogge~\cite{GLS97} (general case). Tataru~\cite{Ta} gave a simpler proof which applies to
$p>p_0(n)$ and all $n\geq 2.$
The critical cases $p=p_0(n)$ in $n\geq 4$ dimensions were
included in the blow-up range by Yordanov $\&$ Zhang~\cite{YZ06} and Zhou~\cite{Z07},
independently. An earlier result of Kato~\cite{Kato80} showed the blow-up when $n=1$ and $p>1$, so
the Strauss conjecture was completely settled in the case of constant coefficients by 2007.

An important problem remained open, however, which was to estimates the lifespan of solutions
when $1<p\leq p_0(n)$ and $\eps\rightarrow 0$.
Let us recall that the \lq\lq lifespan" $T_\e=T_\eps(u_0,u_1)$ is the supremum of all $T>0$,
such that a solution exists to problem (\ref{ivp}) satisfying $(u,u_t)\in C([0,T),H^1(\Real^n)\times L^2(\Real^n))$.
To state the known results when $g_{ij}(x)=\delta_{ij}$, we use the standard notation
$A\sim B$ meaning that there exist positive constants $c$ and $C$, independent of $\e$,
such that $cB\le A\le CB$ holds. It is expected that the exact lifespan estimates for small $\e$
are similar to
\begin{eqnarray*}
T_\e \sim \e^{-2p(p-1)/\gamma(p,n)} &\hbox{if}&  1<p<p_0(n) \hbox{ and }  n\ge3,\\
                                    &  \hbox{or} & 2<p<p_0(2) \hbox{ and } n=2;\\
T_\e \sim\exp(K\e^{-p(p-1)}) &\hbox{if} & p=p_0(n).
\end{eqnarray*}
For low dimensions $n=2$ and $3$, Zhou \cite{Z93}, \cite{Z92_three} and Lindblad \cite{L90}
obtained such results when $1<p<p_0(n)$.
Zhou \cite{Z93}, \cite{Z92_three} also studied the critical case $p=p_0(n)$.
For higher dimensions $n\ge4$, Lai $\&$ Zhou \cite{LZ14} established the lower bound of $T_\eps$
when $1<p<p_0(n)$. The critical case was studied by Lindblad $\&$ Sogge \cite{LS96} who showed
the lower bound of lifespan when $n\le 8$ or initial data are radially symmetric.
Upper bounds on $T_\eps $ were obtained by Takamura \cite{Ta15} in the subcritical case and
by Takamura $\&$ Wakasa \cite{TW11} in the critical case.
Later, Zhou $\&$ Han \cite{ZH14} gave an alternative proof of \cite{TW11}
which also applies to $n=2$ and $3$.

The Strauss' conjecture and lifespan estimates have recently been extended to semilinear wave equations
in other settings, including exterior domains, asymptotically Euclidean spaces, Schwarzschild and Kerr spacetimes.

Let us first review global existence results for the initial boundary value problem
in exterior domains, which require certain local energy decay or non-trapping boundaries.
For supercritical $p>p_0(n)$, Du $\&$ Metcalfe $\&$ Sogge $\&$ Zhou~\cite{DMSZ08} showed global existence when $n=4$ and Hidano $\&$ Metcalfe $\&$ Smith $\&$ Sogge $\&$ Zhou~\cite{HMSSZ10} generalized their result to $n=3$ and $4$ later.
Smith $\&$ Sogge $\&$ Wang \cite{SSW12} proved global existence in the two-dimensional case when $p>p_0(2)$. The blow-up part was verified by Zhou $\&$ Han \cite{ZH11}, together with the upper bound on $T_\eps$, when
$1<p<p_0(n)$ and $n\ge3$.
The critical case $p=p_0(3)$ in $n=3$ was obtained by Lai $\&$ Zhou \cite{LZ15}.
For two-dimensional exterior domains, blow-up results were obtained by
Li $\&$ Wang \cite{LW12}, when $1<p<p_0(2)$, and Lai $\&$ Zhou \cite{LZ17} when $p=p_0(2).$
Lai $\&$ Zhou also proved in \cite{LZ16} that $p=p_0(n)$ belongs to the blow-up range when $n\ge5$.
Concerning lower bounds on the lifespan $T_\eps$, these were studied by Yu \cite{Y11}, in the case $1<p<p_0(3)$,
and by Zha $\&$ Zhou \cite{ZZ15}, in the critical case $p=p_0(4)$.

Next, we turn to results for asymptotically Euclidean space, Schwarzschild and Kerr spacetimes.
The global existence in asymptotically Euclidean spaces was obtained by Wang $\&$ Yu~\cite{WY11}
and Metcalfe $\&$ Wang~\cite{MW} for $p>p_0(n)$ and $n=3,\ 4$.
Moreover, Wang~\cite{W} showed global existence when  $n\ge4$ and $p=2$ and derived sharp lifespan estimates
when $1<p\leq p_0(n)$ and $n=3,\ 4$. The blow-up result for Schwarzschild spacetime was obtained by Catania $\&$ Georgiev~\cite{CG06}
when $n=3$ and $1<p<p_0(3)$. For both Schwarzschild and Kerr spacetimes with small angular momentum,
Lindblad $\&$ Metcalfe $\&$ Sogge $\&$ Tohaneanu $\&$ Wang~\cite{LMSTW14}
showed global existence in the supercritical case $p>p_0(3)$.

This paper contributes to the blow-up part of Strauss' conjecture. We observe that the approach of \cite{Z07} and \cite{ZH14} works for problems in all dimensions and settings if the counterparts of their $\phi_q$ are available.
Here we construct such test functions, which are special solutions to the linear wave equation, using exponential ``eigenfunctions" of $\Delta_g$. Another improvement is the simple blow-up functional, which is just (\ref{Gq}) below.
Unfortunately, we derive a nonlinear integral inequality that is more difficult to study
than the nonlinear differential inequalities appearing in the approach of \cite{Kato80} and \cite{G81a}.
We need the iteration method of \cite{J79},
in its stronger form developed by \cite{AKT00}, to show finite time blow-up and derive sharp lifespan estimates.

\begin{thm}
\label{thm:main}
Let $n\ge2$ and $p=p_0(n).$
Assume that both $u_0\in H^1(\R^n)$ and $u_1\in L^2(\R^n)$ are nonnegative, do
not vanish identically and have supports in the ball $\{x\in\R^n:\ |x|\le R_0 \}$,
where $R_0>1$.

If (\ref{ivp}) has a solution $(u,u_t)\in C([0,T_\eps), H^1(\Real^n)\times L^2(\Real^n))$, such that
\begin{equation}
\label{support}
\mbox{\rm supp}(u,u_t) \subset \{(x,t)\in\R^n\times[0,T_\eps)\ :\ |x|\le t+R\},
\end{equation}
with $R\geq R_0$, then $T_\eps<\infty$. Moreover,
there exist constants $\e_0=\e_0(u_0,u_1,n,p,R)$ and $K=K(u_0,u_1,n,p,R)$, such that
\begin{equation}
\label{thm:lifespan}
T_\e\le\exp\left(K\e^{-p(p-1)}\right)\ \hbox{ for }\ 0<\e\le\e_0.
\end{equation}
%where $C$ is a constant independent of $\e$.
\end{thm}
\begin{rem}
\label{rem111}
The local well-posedness in $H^1(\Real^n)\times L^2(\Real^n)$ is actually given by Brenner~\cite{Br1}.
Our assumptions on the support of solutions can also be verified by Theorem~8 in 7.2 of Evans~\cite{Ev1}.
\end{rem}

To establish Theorem~\ref{thm:main}, we are guided by Zhou~\cite{Z07} and Zhou $\&$ Han~\cite{ZH14}.
Their method introduces and estimates averages of products
$u\phi_q,$ where $\phi_q$ is a smooth positive solution to $(\partial_t^2-\Delta_g)\phi_q=0$ with
behavior as $t-|x|\rightarrow\infty$ determined by a parameter $q$. Fujita~\cite{Fu} also studies
the blow-up problem for nonlinear reaction diffusion equations through averages
with test functions solving the conjugate linear equation. Basically, equation (\ref{ivp}) is multiplied by $\phi_q$
and, after integration by parts and H\"{o}lder's inequality,
a nonlinear differential or integral inequality is derived for
\begin{equation}
\label{Gq}
\int u(x,t)\phi_q(x,t)\: dx.
\end{equation}
Then $q=q(n,p)$ is chosen to optimize the lower bound on this functional. Finite time
blow-up and lifespan estimates are obtained by either a comparison theorem (for differential inequalities)
or an iteration argument (for integral inequalities).

Our proof follows the above steps, although details and notations in Sections 3--5 are slightly different.
An interesting fact is that $\phi_q(x,t)$ with the typical behavior can be constructed even in the case of generalized Laplacian: for any $\lambda_0\in (0,\beta_0)$,
\[
\phi_q (x,t)  = \int_{0}^{\lambda_0}e^{-\lambda(t+R_0)} \varphi_\lambda(x)\lambda^{q-1} \: d\lambda, \quad |x|\leq t+R_0,
\]
where $\vp_\lambda$ is a smooth positive solution to $\Delta_g \varphi_\lambda = \lambda^2\varphi_{\lambda}$, such that
\[
\varphi_\lambda(x) \sim \int_{\S^{n-1}}e^{\lambda x\cdot\omega}dS_{\omega}\sim c_n(\lambda |x|)^{-(n-1)/2}e^{\lambda |x|}, \quad \lambda|x|\rightarrow\infty.
\]

The rest of this paper is organized as follows. In Section~2, we construct $\vp_\lambda(x)$ and study its asymptotics at large $|x|$ and small $\lambda$.
The analog of $\phi_q(x,t)$ is defined and estimated in Section~3.
In Section~4, we derive a nonlinear integral inequality to be used in Section~5 for the proof of Theorem~\ref{thm:main}.

\section{Elliptic equation with small parameter}

Here we will find smooth positive solutions to the elliptic ``eigenvalue problem"
\begin{equation}
\label{eigen-pr}
\Delta_g \varphi_\lambda = \lambda^2\varphi_{\lambda},\quad x\in \Real^n,
\end{equation}
where $\lambda\in (0,\beta/2].$ As $\lambda|x|\rightarrow\infty$, these $\vp_\lambda(x)$ are asymptotically given by $\varphi(\lambda x),$ with $\vp$ being the standard radial solution to the unperturbed equation $\Delta \vp =\vp$:
\begin{equation}
\label{eigen}
\varphi(x)=\int_{\S^{n-1}}e^{x\cdot\omega}dS_{\omega}\sim c_n |x|^{-(n-1)/2}e^{|x|},
\quad |x|\rightarrow\infty.
\end{equation}

The proof relies on the following classical local estimate for weak solutions to
\begin{equation}
\label{elliptic}
(-\Delta_g+\lambda^2)v=f,\quad x\in\R^n.
\end{equation}

\begin{lem}
\label{lem0}
 Assume that $n\geq 2$, $\lambda>0$ and $\Delta_g$ satisfies (\ref{g}) and (\ref{g1}).
  For $q>n$, let $f\in L^{q/2}(\R^n)$ and $v\in H^{1}(\Real^n)$ be the unique weak solution to (\ref{elliptic}).
  Given $y\in\R^n$ and $\rho\in [1,2]$, denote also $B_y(\rho)=\{x\in \R^n:|x-y|\leq \rho\}$.

  Under these assumptions, for any $r>1$,
\begin{eqnarray*}
\|v\|_{L^\infty(B_y(\rho))}& \leq & C \left(\rho^{-n/r}\|v\|_{L^r(B_y(2\rho))}+ \rho^{2(1-n/q)}\|f\|_{L^{q/2}(\R^n)}\right),
\end{eqnarray*}
where $C$ depends only on $p$, $r$, $\rho$, $n$ and the coefficients of $g$.
\end{lem}
This is a special case of Theorem~8.17 in Gilbarg and Trudinger~\cite{GT}.
We obtain the main result of this section as a simple application.

\begin{lem}
\label{lem1}
Let $n\geq 2$. There exists a solution $\varphi_\lambda\in C^\infty(\Real^n)$ to
(\ref{eigen-pr}), such that
\begin{equation}
\label{lem1:ineq}
|\varphi_\lambda(x)-\varphi(\lambda x)|\leq C_\beta\lambda^{\theta},\quad x\in \Real^n,
\quad \lambda\in (0,\beta/2],
\end{equation}
where $\theta\in (0,1]$ and
$\varphi(x)=\int_{\S^{n-1}}e^{x\cdot\omega} dS_\omega\sim c_n|x|^{-(n-1)/2}e^{|x|},$
$c_n>0,$ as $|x|\rightarrow\infty.$

Moreover, $\vp_\lambda(\: \cdot\:)-\vp(\lambda \: \cdot)$ is a continuous $L^\infty (\Real^n)$ valued
function of $\lambda\in (0,\beta/2]$ and there exist positive constants $D_0,$ $D_1$ and $\lambda_0$, such that
\begin{equation}
\label{2sided}
D_0 \langle \lambda|x|\rangle^{-(n-1)/2}e^{\lambda|x|}\leq \varphi_\lambda(x)\leq D_1
\langle \lambda|x|\rangle^{-(n-1)/2}e^{\lambda|x|}, \quad x\in\Real^n,
\end{equation}
holds whenever $0<\lambda\leq \lambda_0$.
\end{lem}

\begin{proof}
Let us choose $\lambda\in (0,\beta/2]$ and consider the elliptic equation
\begin{equation}
\label{psi11}
(-\Delta_g +\lambda^2)\psi_\lambda(x) =f_\lambda(x),
\end{equation}
with $f_\lambda(x)=(\Delta_g-\Delta)\vp(\lambda x)$. There exists a unique solution $\psi_\lambda\in H^{1}(\Real^n)$.
Then $\varphi_\lambda(x):=\varphi(\lambda x)+\psi_\lambda(x)$ satisfies (\ref{eigen-pr}), since
$\Delta\varphi(\lambda x)=\lambda^2\varphi(\lambda x)$. To estimate $\psi_\lambda$ for small $\lambda>0,$
 we take the inner product of (\ref{psi11}) with $\psi_\lambda$ and obtain
\begin{eqnarray*}
\langle g \nabla \psi_{\lambda}, \nabla \psi_{\lambda} \rangle+\lambda^2\|\psi_\lambda\|_{L^2}^2
&=&\langle \nabla (g-I)\nabla \varphi(\lambda x), \psi _{\lambda}\rangle \\
&=&-\langle (g-I)\nabla \varphi(\lambda x), \nabla \psi _{\lambda}\rangle\\
&\leq& C\lambda \|\nabla \psi_{\lambda}\|_{L^2}.
\end{eqnarray*}
The last inequality follows from $\lambda\le \beta/2$ and the fact that $g$ satisfies (\ref{g1}).
We can use (\ref{g}) to further derive
$\gamma \|\nabla \psi_\lambda\|_{L^2}^2+\lambda^2\|\psi_\lambda\|_{L^2}^2
\le C\lambda \|\nabla \psi_{\lambda}\|_{L^2}$, which implies both
\[
\|\nabla \psi_\lambda\|_{L^2}\le C\lambda \ \hbox{ and } \ \|\psi_\lambda\|_{L^2}\le C.
\]
The Gagliardo-Nirenberg inequality allows us to bound also the intermediate norms:
\[
\|\psi_\lambda \|_{L^r}\le C\|\psi_\lambda \|_{L^2}^{1-\theta}\|\nabla \psi_\lambda \|_{L^2}^{\theta},
\]
if $\theta=n(1/2-1/r)\in [0,1]$. We fix $r>2$, such that $\theta>0$, and get
$
\|\psi_{\lambda}\|_{L^r}\le C\lambda^{\theta}.
$
This estimate is substituted into Lemma~\ref{lem0} with $\rho=2$:
\begin{eqnarray*}
\|\psi_\lambda\|_{L^\infty(B_y(2))} & \leq  & C\left(\|\psi_\lambda\|_{L^r(B_y(4))}+\|f_\lambda\|_{L^{q/2}(\R^n)}\right)\\
                                    & \leq & C(\lambda^\theta+\lambda+\lambda^2).
\end{eqnarray*}
Hence, $\|\psi_\lambda \|_{L^\infty(B_y(2))}\leq C\lambda^{\theta}$,
where $C$ is independent of $y$. This bound holds for every $y$ with integer coordinates, so we conclude that $\|\psi_\lambda \|_{L^\infty(\R^n)}\leq C\lambda^{\theta}$,
which is the desired estimate. Finally, we combine $\Delta_g \varphi_\lambda = \lambda^2\varphi_{\lambda}$ and
$\vp_\lambda\in L_{loc}^\infty(\R^n)$ with Theorem~8.10 in \cite{GT} to obtain that $\vp_\lambda\in C^\infty(\R^n)$.

It remains to show that $\psi_\lambda(x)=\vp_\lambda(x)-\vp(\lambda x)$ is a continuous $L^\infty(\Real^n)$ valued function of $\lambda$. We consider the equation for $\psi_{\lambda}-\psi_\nu$, where $\lambda,\: \nu\in (0,\beta/2]$:
\[
(-\Delta_g +\lambda^2)(\psi_\lambda-\psi_\nu) =(f_\lambda-f_\nu) +(\nu^2-\lambda^2)\psi_\nu.
\]
The inner product with $\psi_\lambda-\psi_\nu$ in $L^2(\Real^n)$ yields the estimate
\begin{eqnarray*}
\langle g \nabla (\psi_\lambda-\psi_\nu), \nabla (\psi_\lambda-\psi_\nu) \rangle
+\lambda^2\|\psi_\lambda-\psi_\nu\|_{L^2}^2 &\leq & C(\lambda)\|f_\lambda-f_\nu\|_{L^2}^2\\
 & & +C(\lambda)|\nu^2-\lambda^2|^2\|\psi_\nu\|_{L^2}^2.
\end{eqnarray*}
Since $\|\psi_\nu\|_{L^2}\leq C$ and 
$f_{\lambda}(x)-f_{\nu}(x)=(\Delta_g-\Delta)(\vp(\lambda x)-\vp(\nu x))$, 
we get 
\[
\gamma^{1/2}\|\nabla (\psi_\lambda-\psi_\nu) \|_{L^2}+\lambda\|\psi_\lambda-\psi_\nu\|_{L^2}\leq 
C_1(\lambda)|\lambda-\nu|.
\]
By the Gagliardo-Nirenberg inequality, $\|\psi_\lambda-\psi_\nu\|
_{L^r}\leq C_2(\lambda)|\lambda-\nu|$ whenever
$n(1/2-1/r)\in [0,1].$ We combine this and Lemma~\ref{lem0} with  $\rho=2$,  $q>\max\{n,4\}$:
\begin{eqnarray*}
\|\psi_\lambda-\psi_\nu \|_{L^\infty(B_y(2))} & \leq  &
  C\|\psi_\lambda-\psi_\nu \|_{L^r(B_y(4))}+C\|f_\lambda-f_\nu\|_{L^{q/2}(\R^n)}\\
        & & +C|\lambda^2-\nu^2|\|\psi_\nu\|_{L^{q/2}(\R^n)}.
\end{eqnarray*}
Then, $ \|\psi_\nu\|_{L^{q/2}(\R^n)}^{q/2} \leq \|\psi_\nu\|_{L^\infty(\R^n)}^{q/2-2} \|\psi_\nu\|_{L^2(\R^n)}^2\leq C_3(\lambda)$, so we have
$$\|\psi_\lambda-\psi_\nu \|_{L^\infty(B_y(2))} \leq C_4(\lambda)|\lambda-\nu|.$$
The independence of $y$ implies $\|\psi_\lambda-\psi_\nu\|_{L^\infty(\R^n)}\leq C_4(\lambda)|\lambda-\nu|\rightarrow 0$ as $ \nu\rightarrow\lambda.$

Finally, (\ref{2sided}) follows from (\ref{lem1:ineq}) and $\vp(0)=\hbox{area}(\S^{n-1})>0.$
\end{proof}

\section{Test functions}

We define and estimate two test functions, solutions of the linear wave equation, which are used to derive
the nonlinear integral inequality (\ref{frame}) in the next section.

For $\lambda_0\in (0,\beta/2]$ and $q>-1$, let
\begin{eqnarray}
\label{aq}
\xi_q (x,t) & = & \int_{0}^{\lambda_0}e^{-\lambda(t+R)}\cosh\lambda t \: \varphi_\lambda(x)\lambda^{q} d\lambda,\\
\eta_q (x,t,s) & = & \int_{0}^{\lambda_0}e^{-\lambda(t+R)}\frac{\sinh\lambda (t-s)}{\lambda(t-s)}\: \varphi_\lambda(x)\lambda^{q} d\lambda,
\label{bq}
\end{eqnarray}
where $(x,t)\in \R^n\times \R$ and $s\in\R.$
In fact, $\eta_q(x,t,t)$ solves $(\partial_{t}^2-\Delta_g)\eta_{q}(x,t,t)=0$ and generalizes
the test function $\phi_{q+1}(x,t)$ introduced in \cite{Z07}.

\begin{lem}
\label{lem2}
Let $n\geq 2$. There exists $\lambda_0\in (0,\beta/2]$, such that the following hold:

(i) if $0<q$, $|x|\leq R$ and $0\leq t$, then
\begin{eqnarray*}
\xi_q (x,t) & \geq & A_0,\\
\eta_q (x,t,0) & \geq & B_0\langle t\rangle^{-1};
\end{eqnarray*}

(ii) if $0<q$, $|x|\leq s+R$ and $0\leq s<t$, then
\begin{eqnarray*}
\eta_q (x,t,s) & \geq & B_1 \langle t\rangle^{-1}\langle s\rangle^{-q};
\end{eqnarray*}

(iii) if $(n-3)/2<q$, $|x|\leq t+R$ and $0<t$, then
\begin{eqnarray*}
\eta_q (x,t,t) & \leq & B_2 \langle t\rangle^{-(n-1)/2}\langle t-|x| \rangle^{(n-3)/2-q}.
\end{eqnarray*}
Here $A_0$ and $B_k$, $k=0,1,2,$ are positive constants depending only on $\beta$, $q$ and $R$,
while $\langle s\rangle =3+|s|$ is used to simplify estimates in Sections 4 and 5.
\end{lem}
\begin{proof}

Claim $(i)$ is evident from (\ref{aq}), (\ref{bq}) and $\inf_{0\leq \lambda\leq \lambda_0}\inf_{|x|\leq R}\vp_\lambda(x)>0$:
\begin{eqnarray*}
\xi_q(x,t)  & \geq & \inf_{0\leq \lambda\leq \lambda_0}\inf_{|x|\leq R}\vp_\lambda(x) \int_{0}^{\lambda_0}e^{-\lambda R}
\frac{1+e^{-2\lambda t}}{2}\lambda^q\: d\lambda
                      \geq A_0,\\
\eta_q (x,t,0) & \geq & \inf_{0\leq \lambda\leq \lambda_0}\inf_{|x|\leq R}
\vp_\lambda(x)\int_{0}^{\lambda_0}e^{-\lambda R}
\frac{1-e^{-2 \lambda t}}{2\lambda t}\lambda^{q}\: d\lambda\geq B_0\langle t\rangle^{-1}.
\end{eqnarray*}

For claim $(ii)$, we combine (\ref{bq}) and the positivity of $\vp_\lambda(x)$ from (\ref{2sided}). Then
\begin{eqnarray*}
\eta_q (x,t,s) & = & \int_{0}^{\lambda_0}e^{-\lambda(t-s)}\frac{\sinh\lambda (t-s)}{\lambda(t-s)}\:
[e^{-\lambda(s+R)}\vp_\lambda(x)]\lambda^q d\lambda\\
            & \geq & \int_{\lambda_0/\langle s\rangle }^{2\lambda_0/\langle s\rangle }\frac{1-e^{-2\lambda (t-s)}}{2(t-s)}\: [e^{-\lambda(s+R)}\vp_\lambda(x)]\lambda^{q-1} d\lambda,
\end{eqnarray*}
since $\langle s\rangle\geq 2$ for all $s\in\R$. We further obtain that
\begin{eqnarray*}
 \eta_q (x,t,s) & \geq & A_1\int_{\lambda_0/\langle s\rangle}^{2\lambda_0/\langle s\rangle}\frac{1-e^{-2\lambda (t-s)}}{2(t-s)}\: \lambda^{q-1} d\lambda,
\end{eqnarray*}
with $A_1\leq \inf_{\lambda_0/\langle s\rangle \leq \lambda\leq 2\lambda_0/\langle s\rangle}
 \inf_{|x|\leq s+R}e^{-\lambda(s+R)}\vp_\lambda(x).$ It follows from (\ref{2sided}) that this
lower bound $A_1>0$ can be chosen independent of $x$ and $s$. We finally have
\begin{eqnarray*}
 \eta_q (x,t,s) & \geq & A_1\frac{1-e^{-\lambda_0(t-s)/\langle s\rangle}}{2(t-s)}\int_{\lambda_0/\langle s\rangle}^{2\lambda_0/\langle s\rangle}\lambda^{q-1} d\lambda\geq B_1 \langle t\rangle^{-1}\langle s \rangle^{-q}.
\end{eqnarray*}

The last claim $(iii)$ follows from the upper bound (\ref{2sided}) substituted into (\ref{bq}):
\begin{eqnarray*}
\eta_q (x,t,t) & \leq & D_0^{-1} \int_0^{\lambda_0} \frac{e^{-\lambda(t+R-|x|)}\lambda^{q} }{\langle\lambda|x|\rangle^{(n-1)/2}}\: d\lambda.
\end{eqnarray*}
It is convenient to consider two cases. If $|x|\leq (t+R)/2$, the estimate becomes
\[
\eta_q (x,t,t) \leq  D_1 \int_0^{\lambda_0} e^{-\lambda(t+R)/2}\lambda^{q} d\lambda\leq D_2 \langle t\rangle^{-q-1}.
\]
If $|x|\geq (t+R)/2$, the resulting bound is different:
\begin{eqnarray*}
\eta_q (x,t,t) & \leq  & D_0^{-1} \langle |x|\rangle^{-(n-1)/2}\int_0^{\lambda_0} e^{-\lambda(t+R-|x|)}\lambda^{q-(n-1)/2}
                      \: d\lambda\\
            & \leq &  D_3\langle |x| \rangle^{-(n-1)/2}\langle t-|x|\rangle^{(n-3)/2-q}.
\end{eqnarray*}
Clearly, both results are included into $\eta_q (x,t,t) \leq B_2 \langle t \rangle^{-(n-1)/2}\langle t-|x|\rangle^{(n-3)/2-q}.$
\end{proof}

%%%%%%%%%%%%%%%%%%%%%%%%%%%%%%%%%%%%%%%%%%%%%%%%%%%%%%%%%%%%%%%%%%%%%%%%%%%%%%%%%%%%
%%%%%%%%%%%%%%%%%%%%%%%%%%%%%%%%%%%%%%%%%%%%%%%%%%%%%%%%%%%%%%%%%%%%%%%%%%%%%%%%%%%%
%%%%%%%%%%%%%%%%%%%%%%%%%%%%%%%%%%%%%%%%%%%%%%%%%%%%%%%%%%%%%%%%%%%%%%%%%%%%%%%%%%%
\section{Nonlinear integral inequality}
We will average the weak solution $u$ of problem (\ref{ivp}) with respect to suitable test functions from
Section 3. In all cases, we take $q>-1$ and consider
\begin{equation}
\label{def:F}
F(t)=\int_{\R^n} u(x,t) \eta_{q}(x,t,t)dx.
\end{equation}
This functional satisfies a nonlinear integral inequality whenever $u$ is an energy space solution:
$(u,u_t)\in C([0,T_\eps), H^1(\Real^n)\times L^2(\Real^n))$ and $\forall \phi\in C_0^\infty(\Real^n\times [0,T_\eps))$
\begin{eqnarray}
\nonumber
& & \int u_s(x,t)\phi(x,t) dx-\int u_s(x,0)\phi(x,0) dx\\
\label{43}
& & -\int_0^t \int (u_s(x,s)\phi_{s}(x,s)- g(x)\nabla u(x,s)\cdot\nabla \phi(x,s)) dxds \\
\nonumber
& & =\int_0^\infty \int |u(x,s)|^p \phi(x,s) dxds,
\end{eqnarray}
for $t\in (0,T_\eps).$ We can actually use $\phi\in C^\infty(\Real^n\times [0,T_\eps))$ in the next result, since
$u(\cdot,s)$ is compactly supported for every $s$.
\begin{prop}
\label{prop:identity}
Let the assumptions in Theorem \ref{thm:main} be fulfilled and $q>-1.$
\begin{equation}
\label{final-equal}
\begin{array}{lll}
\d \int_{\R^n}u(x,t) \eta_{q}(x,t,t) dx\!\!\!& = &\!\!\!
\d \e\int_{\R^n}\!\!u_0(x)\xi_{q}(x,t)\: dx+ \e t\int_{\R^n}\!\!u_1(x) \eta_{q}(x,t,0) dx\\
& &\d +\int_0^t(t-s) \int_{\R^n}|u(x,s)|^p \eta_{q}(x,t,s) dxds
\end{array}
\end{equation}
for all $t\in (0,T_\eps),$ where $\xi_q$ and $\eta_q$ are defined in (\ref{aq}) and (\ref{bq}).
\end{prop}
\begin{proof} It is convenient to integrate by parts in (\ref{43}) and obtain that
\begin{eqnarray*}
\nonumber
& & \int (u_s(x,t)\phi(x,t)-u(x,t)\phi_s(x,t))dx-\int (u_s(x,0)\phi(x,0)-u(x,0)\phi_s(x,0))dx\\
\nonumber
& & +\int_0^t \int u(x,s)(\phi_{ss}(x,s)- \Delta_g\phi(x,s))dxds \\
\nonumber
& & =\int_0^\infty \int |u(x,s)|^p \phi(x,s) dxds.
\end{eqnarray*}
We choose $\phi(x,s)=\vp_\lambda(x)\lambda^{-1}\sinh \lambda(t-s)$, which solves $\phi_{ss}(x,s)- \Delta_g\phi(x,s)=0$,
and use the initial conditions in (\ref{ivp}) to derive
\begin{eqnarray*}
\d\int_{\R^n}u(x,t)\varphi_{\lambda}(x) dx & = &\e \cosh (\lambda t)\int_{\R^n}u_0(x)\varphi_{\lambda}(x) dx\\
 & & +\d\e\: \frac{\sinh(\lambda t)}{\lambda}\int_{\R^n}u_1(x)\varphi_{\lambda}(x)dx\\
& & +\int_{0}^{t}\frac{\sinh(\lambda (t-s))}{\lambda}\left(\int_{\R^n}|u(x,s)|^p\varphi_{\lambda}(x) dx\right)ds.
\end{eqnarray*}
The desired identity follows, if we multiply through by $\lambda^q e^{-\lambda(t+R)}$, integrate on $[0,\lambda_0]$ and interchange the order of integration between $\lambda$ and  $x$.
Recalling that $\xi_q$ and $\eta_q$ are defined by (\ref{aq}) and (\ref{bq}), respectively, we complete the proof.
\end{proof}

From now on, we use $C$ to denote positive constants depending only on $n$, $p$, $q$ and $R$, which
may change from line to line. The following proposition is the frame of our iteration argument which
shows the finite time blow-up of $u$ and yields an asymptotically sharp estimate of $T_\eps$ as $\eps\rightarrow 0.$
\begin{prop}
\label{prop:frame}
Suppose that the assumptions in Theorem \ref{thm:main} are fulfilled and choose
$q=(n-1)/2-1/p.$ If $F(t)$ is defined in (\ref{def:F}), there exists a positive constant $C=C(n,p,R)$, such that
\begin{equation}
\label{frame}
F(t)  \geq \frac{C}{\langle t\rangle}
\int_0^t \frac{t-s}{\langle s\rangle}\frac{F(s)^p}{
(\log \langle s\rangle)^{p-1}}\: ds
\end{equation}
for all $t\in (0,T_\eps).$
\end{prop}

\begin{proof}
Let $0\leq s<t$. From $F(s)=\int_{\R^n} u(x,s) \eta_{q}(x,s,s)dx$ and H\"{o}lder's inequality,
\begin{equation}
\label{Holder:F1}
\begin{array}{lll}
\d |F(s)|\leq \left(\int |u(x,s)|^p \eta_{q}(x,t,s) dx \right)^{1/p} \\
\d\qquad \times \left( \int_{|x|\leq s+R} \frac{\{\eta_{q}(x,s,s)\}^{p/(p-1)}}
{\{\eta_{q}(x,t,s)\}^{1/(p-1)}}
dx\right)^{(p-1)/p}.
\end{array}
\end{equation}
Substituting estimates $(ii)$ and $(iii)$ from Lemma~\ref{lem2} with $q=(n-1)/2-1/p$,
we can bound the second integral by
\[
C\int_{|x|\leq s+R} \frac{\langle s\rangle ^{-(n-1)p/2(p-1)} \langle s-|x| \rangle ^{\{(n-3)/2-q\}p/(p-1)}}
{\langle t\rangle ^{-1/(p-1)}\langle s \rangle ^{-q/(p-1)}}dx.
\]
This expression simplifies to
\begin{eqnarray*}
C\langle t\rangle ^{1/(p-1)}\langle s \rangle ^{q/(p-1)-(n-1)p/2(p-1)}
\int_{|x|\leq s+R} \langle s-|x| \rangle ^{\{(n-3)/2-q\}p/(p-1)} dx.
\end{eqnarray*}
The latter integral is actually
\[
\int_{|x|\leq s+R} \langle s-|x| \rangle ^{-1} dx\leq C\langle s\rangle^{n-1}\log \langle s\rangle,
\]
so the final estimate of the second integral in (\ref{Holder:F1}) becomes
\[
C\langle t\rangle ^{1/(p-1)}
 \langle s\rangle ^{(n-1)/2-1/p(p-1)}\log\langle s\rangle.
\]
From Lemma~\ref{lem2} and Proposition~\ref{prop:identity}, we see that $F(t)\geq 0.$ Thus, (\ref{Holder:F1}) gives
\[
F(s)^p\leq C\langle t\rangle  \langle s \rangle^{(n-1)(p-1)/2-1/p} (\log\langle s\rangle)^{p-1}
\left(\int |u(x,t)|^p \eta_{q}(x,t,s) dx\right).
\]
Combining Proposition~\ref{prop:identity} with estimates $(i)$ in Lemma \ref{lem2}, we have that
\begin{equation}
\label{frame-sub}
F(t)  \geq   C_1(u_0)\e+C_2(u_1)\e t\langle t\rangle^{-1}
 +\frac{C}{\langle t\rangle}\int_0^t\frac{ (t-s)F(s)^p\: ds}{ \langle s\rangle ^{(n-1)(p-1)/2-1/p}
(\log\langle s\rangle)^{p-1}}.
\end{equation}
Since $p=p_0(n)$ is equivalent to
%\begin{equation}
%\label{critical-relation}
\[
\frac{n-1}{2}p-\frac{n-1}{2}-\frac{1}{p}=1,
\]
%\end{equation}
inequality (\ref{frame-sub}) implies (\ref{frame}). The proof is complete.
\end{proof}

%%%%%%%%%%%%%%%%%%%%%%%%%%%%%%%%%%%%%%%%%%%%%%%%%%%%%%%%%%%%%%%%%%%%%%%%%%%
%%%%%%%%%%%%%%%%%%%%%%%%%%%%%%%%%%%%%%%%%%%%%%%%%%%%%%%%%%%%%%%%%%%%%%%%%%%
\section{Iteration argument}

First of all, we shall get the first step of the iteration argument. To obtain estimate (\ref{1ststep}), we use the
following lower bound of the $L^p$ norm of $u$.
\begin{lem}
\label{lem:lp_lower}
Suppose that the assumptions in Theorem \ref{thm:main} are fulfilled.
Then, there exists a positive constant $C_0=C_0(u_0,u_1,n,p,R)$ such that
\begin{equation}
\label{lp_lower}
\int_{\Real^n}|u(t,x)|^pdx\ge C_0 \e^p \langle t\rangle ^{n-1-(n-1)p/2}.
\end{equation}
holds for $t\ge0$.
\end{lem}
\begin{proof}
Making use of (\ref{frame-sub}) and H\"{o}lder's inequality, we get
\begin{equation}
\label{Holder_u^p}
C_1(u_0,u_1)\e \le |F(t)| \le \left(\int_{\R^n}|u(x,t)|^pdx\right)^{1/p}
\cdot (I(t))^{1/p'},
\end{equation}
where we set
\[
I(t)=\int_{|x|\le t+R} \{\eta_{q}(x,t,t)\}^{p'}dx,
\]
and $p'=p/(p-1)$.
It follows from the estimates (iii) in Lemma \ref{lem2} with $q>(n-3)/2+1/p'$ that
\begin{eqnarray*}
I(t)&\le& C\langle t\rangle^{-(n-1)p'/2} \int_{|x|\le t+R}\langle t-|x|\rangle^{(n-3)p'/2-p'q}dx
\\ &=& C\langle t\rangle^{-(n-1)p'/2}\int_{0}^{t+R} r^{n-1} \langle t-r\rangle^{(n-3)p'/2-p'q} dr
\end{eqnarray*}
Changing the variables by $t-r=\rho$, we have
\begin{eqnarray*}
I(t)&\le& C\langle t\rangle^{-(n-1)p'/2} \int_{-R}^{t}(t-\rho)^{n-1}(3R+|\rho|)^{(n-3)p'/2-p'q}d\rho\\
&=& C\langle t\rangle^{-(n-1)p'/2}\{I_1(t)+I_2(t)\},
\end{eqnarray*}
where we set
\[
I_1(t)=\int_{-R}^{t/2} (t-\rho)^{n-1} (3R+|\rho|)^{(n-3)p'/2-p'q} d\rho
\]
and
\[
I_2(t)=\int_{t/2}^{t} (t-\rho)^{n-1} (3R+\rho)^{(n-3)p'/2-p'q} d\rho.
\]
Since $(n-3)p'/2-p'q+1<0$, integration by parts yields that
\begin{eqnarray*}
I_2(t)&\le& C\langle t\rangle^{n-1} (3R+t)^{(n-3)p'/2-p'q+1} \\
&&\qquad-C\int_{t/2}^{t} (t-\rho)^{n-2} (3R+\rho)^{(n-3)p'/2-p'q+1}d\rho\\
&\le& C\langle t\rangle^{n-1+(n-3)p'/2-p'q+1}.
\end{eqnarray*}
Similarly, we have
\begin{eqnarray*}
I_1(t)&\le& C(t+R)^{n-1}-
\int_{-R}^{t/2}(t-\rho)^{n-2}(3R+\rho)^{(n-3)p'/2-p'q+1}d\rho\\
&\le& C(t+R)^{n-1}\le C\langle t\rangle^{n-1}.
\end{eqnarray*}
Therefore, we get
\[
I(t)\le C\langle t\rangle^{n-1-(n-1)p'/2}
\]
which implies (\ref{lp_lower}) by (\ref{Holder_u^p}).
\end{proof}

In the following, we start our iteration argument by using the
\lq\lq slicing method" in \cite{AKT00}.
Let us show the first step of the iteration argument.
\begin{lem}
\label{lem:first-it}
Suppose that the assumptions in Theorem \ref{thm:main} are fulfilled.
Then, $\d F(t)=\int_{\R^n} u(x,t) \eta_{q}(x,t,t)dx$ for $t\ge 3/2$ satisfies that
\begin{equation}
\label{1ststep}
F(t)\ge M\e^p \log \{t/(3/2)\},
\end{equation}
where $M=C_0B_1/3^3$ and $C_0$ is the one in Lemma \ref{lem:lp_lower}.
\end{lem}

\begin{proof}

Putting the estimates (\ref{lp_lower}) and (ii) with $q=(n-1)/2-1/p>0$ in Lemma \ref{lem2} into (\ref{final-equal}), we
get
\[
F(t)\ge \frac{C_0B_1\e^p}{\langle t\rangle}\int_{0}^{t}\frac{t-s}
{\langle s\rangle^{q+(n-1)p/2-(n-1)}}ds.
\]
Let $t\ge 3/2$. It follows from
\[
q+\frac{(n-1)p}{2}-(n-1)=\frac{(n-1)p}{2}-\frac{(n-1)}{2}-\frac{1}{p}=1.
\]
that
\begin{eqnarray*}
F(t)&\ge&\frac{C_0B_1\e^p}{3^2 t}\int_{1}^{t}\frac{t-s}{s}ds
\ge \frac{C_0B_1\e^p}{3^2 t}\int_{2t/3}^{t}\log s ds\\
&\ge&\frac{C_0B_1\e^p}{3^3} \log (2t/3)
\end{eqnarray*}
for $t\ge 3/2$. The proof is complete.
\end{proof}

%Our proposition argument will be done by using the following estimates.
The next step is to derive the following estimates.

\begin{prop}
\label{prop:j-step}
Suppose that the assumptions in Theorem \ref{thm:main} are fulfilled.
Then, $\d F(t)=\int_{\R^n} u(x,t) \eta_{q}(x,t,t)dx$ for $t\ge l_j $ $(j\in \N)$ satisfies that
\begin{equation}
\label{j-step}
F(t)\ge C_j (\log \langle t\rangle)^{-b_j}\left\{ \log \left(t/l_j \right)\right\}^{a_j},
\end{equation}
where $\d l_j=l_0+\sum_{k=1}^{j}2^{-(k+1)}=2-2^{-(j+1)}$ $(j\in\N)$ with $l_0=3/2$.
Here, $a_j$, $b_j$ and $C_j$ are defined by
\begin{equation}
\label{a_j,b_j}
a_j=\frac{p^{j+1}-1}{p-1}\quad \mbox{and}\quad b_j=p^{j}-1,
\end{equation}
\begin{equation}
\label{ind_C_j}
C_j=\exp\{p^{j-1}(\log(C_1(2p)^{-S_j}E^{1/(p-1)})-\log E^{1/(p-1)})\}\quad (j\ge 2),
\end{equation}
\begin{equation}
\label{C_1}
C_1=N\e^{p^2},
\end{equation}
where $C$ is the one in (\ref{frame}) and
\begin{equation}
\label{S_j,E}
N=\frac{C M^p}{3^27(p+1)},\quad S_j=\sum_{i=1}^{j-1}\frac{i}{p^i},\quad E=\frac{C(p-1)}{2^33^2p^2}.
\end{equation}
\end{prop}
\begin{proof}

Let $t\ge l_1$.
Replacing the domain of integration by $[l_0,t]$ in (\ref{frame})
and putting the estimates (\ref{1ststep}) into (\ref{frame}), we get
\begin{eqnarray*}
\d F(t)
&\ge& \frac{CM^p\e^{p^2}}{\langle t\rangle} \int_{l_0}^{t}
\frac{(t-s) \{\log(s/l_0)\}^{p}}
{\langle s\rangle(\log \langle s\rangle)^{p-1}}ds\\
\d &\ge&\frac{CM^p\e^{p^2}}{3^2t}(\log \langle t\rangle)^{-(p-1)}
\int_{l_0}^{t}\frac{(t-s)}{s}\{\log(s/l_0)\}^{p}ds
\end{eqnarray*}
for $t\ge l_{1}$.
Integration by parts yields that
\[
F(t)
\ge \frac{C M^p\e^{p^2}}{3^2(p+1)t}(\log \langle t\rangle)^{-(p-1)}
\int_{l_0}^{t}\{\log(s/l_0)\}^{p+1}ds.
\]
Replacing the domain of integration by $[l_0t/l_1,t]$, we get
\begin{eqnarray*}
F(t)&\ge&\frac{C M^p\e^{p^2}}{3^2(p+1)t}(\log \langle t\rangle)^{-(p-1)}
\int_{l_0t/l_1}^{t}\{\log(s/l_0)\}^{p+1}ds\\
&\ge& \frac{C M^p\e^{p^2}}{3^27(p+1)}(\log \langle t\rangle)^{-(p-1)}\{\log(t/l_{1})\}^{p+1}
\end{eqnarray*}
for $t\ge l_1$. Therefore, (\ref{j-step}) holds for $j=1$.

Assume that (\ref{j-step}) holds. Let $t\ge l_{j+1}$.
Replacing the domain of integration by $[l_j,t]$ in (\ref{frame}) and
putting the estimates (\ref{j-step}) into (\ref{frame}), we get
\begin{eqnarray*}
\d F(t)
&\ge& \frac{C C_j^p}{\langle t\rangle} \int_{l_j}^{t}
\frac{(t-s)(\log \langle s\rangle)^{-pb_j}\{\log(s/l_j)\}^{pa_j}}
{\langle s\rangle(\log \langle s\rangle)^{p-1}}ds\\
\d &\ge&\frac{C C_j^p}{3^2t}(\log \langle t\rangle)^{-pb_j-(p-1)}
\int_{l_j}^{t}\frac{(t-s)}{s}\{\log(s/l_j)\}^{pa_j}ds
\end{eqnarray*}
for $t\ge l_{j+1}$.
Integration by parts yields that
\[
F(t)
\ge \frac{C C_j^p}{3^2(pa_j+1)t}(\log \langle t\rangle)^{-pb_j-(p-1)}
\int_{l_j}^{t}\{\log(s/l_j)\}^{pa_j+1}ds.
\]
Replacing the domain of integration by $[l_jt/l_{j+1},t]$, we get
\begin{eqnarray*}
F(t)
&\ge& \frac{C C_j^p}{3^2(pa_j+1)t}(\log \langle t\rangle)^{-pb_j-(p-1)}
\int_{l_jt/l_{j+1}}^{t}\{\log(s/l_j)\}^{pa_j+1}ds\\
&\ge&\frac{C C_j^p(1-l_j/l_{j+1})}{3^2(pa_j+1)}(\log \langle t\rangle)^{-pb_j-(p-1)}\{\log(t/l_{j+1})\}^{pa_j+1}
\end{eqnarray*}
for $t\ge l_{j+1}$.
Noticing that $1-l_j/l_{j+1}=(l_{j+1}-l_j)/l_{j+1}\ge2^{-(j+3)}$ and the definitions of $\{a_j\}, \{b_j\}$ in
(\ref{a_j,b_j}), we obtain
\begin{eqnarray*}
F(t)&\ge&
\frac{CC_j^p}{2^{j+3}3^2(pa_j+1)}(\log \langle t\rangle)^{-b_{j+1}}\{\log(t/l_{j+1})\}^{a_{j+1}}\\
&\ge& \frac{EC_j^p}{(2p)^{j}}(\log \langle t\rangle)^{-b_{j+1}}\{\log(t/l_{j+1})\}^{a_{j+1}}
\end{eqnarray*}
for $t\ge l_{j+1}$, where $E$ is defined in (\ref{S_j,E}).
Finally, it remains to prove that $C_j$ in (\ref{ind_C_j}) satisfies
\[
C_{j+1}=\frac{EC_j^p}{(2p)^{j}}.
\]
Since Lemma 3.1 in \cite{W15}, if $C_{a,j}$, $F_{p,a}$, $E_{p,a}$
are replaced by $C_j$, $2p$, and $E$, respectively, we have this equality.
The proof is complete.
\end{proof}

%\par\noindent
%{\bf End of the proof of Theorem \ref{thm:main}.}

\begin{proof}[End of the proof of Theorem \ref{thm:main}] 
Setting $\d S = \lim_{j\rightarrow \infty}S_j$,
we see $S_j \le S$ for all $j\in \N$.
Therefore, (\ref{ind_C_j}) yields
\begin{equation}
\label{fin-ind_C_j}
\begin{array}{llll}
C_j&\ge& \exp\{p^{j-1}(\log(C_1(2p)^{-S}E^{1/(p-1)})-\log E^{1/(p-1)})\}\\
&\ge& E^{-1/(p-1)} \exp\{p^{j-1}(\log(C_1(2p)^{-S}E^{1/(p-1)}))\}.
\end{array}
\end{equation}
Combing the estimates (\ref{fin-ind_C_j}) with (\ref{j-step}), we have
\begin{eqnarray*}
F(t)\ge
E^{-1/(p-1)} \exp\{p^{j-1}\{\log (C_1(2p)^{-S}(\log(3+t))^{-p}(\log(t/2))^{p^2/(p-1)})\}\}\\
\times\log(3+t)\{\log (t/2)\}^{-1/(p-1)}.
\end{eqnarray*}
for $t\ge 2$.
Noticing that $\log(2t)\le 2\log t$, $\log(t/2)\ge (\log t)/2 $ for $t\ge4$, we get
\[
F(t)\ge
E^{-1/(p-1)} \exp\{p^{j-1}K(t)\}\log(3+t)\{\log (t/2)\}^{-1/(p-1)},
\]
where $K(t)=\log \{B\e^{p^2}(\log t)^{p/(p-1)}\}$ with $B=N(2p)^{-S}2^{p(1-2p)/(p-1)}E^{1/(p-1)}$
for $t\ge4$.

We take $\e_0=\e_0(u_0,u_1,p,R,n)>0$ so small that
\[
\exp\{B^{-(p-1)/p}\e_0^{-p(p-1)}\} \ge 4.
\]
Next, for a fixed $\e\in(0,\e_0]$, we suppose that $T$ satisfies
\begin{equation}
\label{det:lifespan}
T>\exp \{B^{-(p-1)/p}\e ^{-p(p-1)}\}\ (\ge4).
\end{equation}
Then we have $K(T)>0$. Therefore, we get $F(T)\rightarrow \infty$ as $j\rightarrow \infty$.
Hence, (\ref{det:lifespan}) implies that $T_\e\le \exp \{B^{-(p-1)/p}\e ^{-p(p-1)}\}$ for $0<\e\le \e_0$.

\end{proof}

\newpage

\begin{center}
ACKNOWLEDGMENTS
\end{center}
The authors are grateful to Professor Hiroyuki Takamura for his useful comments.

\end{document}